\documentclass[onefignum,onetabnum]{siamonline181217}

\usepackage{lipsum}
\usepackage{amsfonts}
\usepackage{graphicx}
\usepackage{epstopdf}
\usepackage{algorithmic}
\ifpdf
  \DeclareGraphicsExtensions{.eps,.pdf,.png,.jpg}
\else
  \DeclareGraphicsExtensions{.eps}
\fi

\usepackage{enumitem}
\setlist[enumerate]{leftmargin=.5in}
\setlist[itemize]{leftmargin=.5in}

\newsiamremark{remark}{Remark}
\newsiamremark{example}{Example}
\newsiamremark{hypothesis}{Hypothesis}
\crefname{hypothesis}{Hypothesis}{Hypotheses}
\newsiamthm{claim}{Claim}

\headers{Multiplicative Relations between Polynomial Roots}{T. Zheng}
\title{Computing Multiplicative Relations between Roots of a Polynomial\thanks{November 28, 2019.
\funding{This work was supported partly by NSFC under grants 61732001
and 61532019.}}}

\author{
Tao Zheng\thanks{School of Mathematical Sciences, Peking University, Beijing, China (\email{xd07121019@126.com}).}
}

\usepackage{amsopn}

\ifpdf
\hypersetup{
  pdftitle={Computing the Exponent Lattice of Generic Polynomial Roots: A Fast Approach
  },
 pdfauthor={T. Zheng}
}
\fi

\usepackage{arydshln}
\usepackage{amssymb}
\usepackage{mathrsfs}
\usepackage{algorithm}
\usepackage{algorithmic}
\usepackage{multirow}
\usepackage[numbers]{natbib}
\begin{document}
\renewcommand\arraystretch{1.2}
\maketitle
\begin{abstract}
Multiplicative relations between the roots of a polynomial in $\mathbb{Q}[x]$ have drawn much attention in the field of arithmetic and algebra, while the problem of computing these relations is interesting to researchers in many other fields. In this paper, a sufficient condition is given for a polynomial $f\in\mathbb{Q}[x]$ to have only trivial multiplicative relations between its roots, which is a generalization of those sufficient conditions proposed in [C. J. Smyth,  \emph{J. Number Theory}, 23 (1986), pp. 243--254], [G. Baron \emph{et al}., \emph{J. Algebra}, 177 (1995), pp. 827--846] and [J. D. Dixon, \emph{Acta Arith.} 82 (1997), pp. 293--302]. Based on the new condition, a subset $E\subset\mathbb{Q}[x]$ is defined and proved to be genetic (i.e., the set $\mathbb{Q}[x]\backslash E$ is very small). We develop an algorithm deciding whether a given polynomial $f\in\mathbb{Q}[x]$ is in $E$ and returning a basis of the lattice consisting of the multiplicative relations between the roots of $f$ whenever $f\in E$. The numerical experiments show that the new algorithm is very efficient for the polynomials in $E$. A large number of polynomials with much higher degrees, which were intractable before, can be handled successfully with the algorithm. 
\end{abstract}

\begin{keywords}
exponent lattice, multiplicative relation, polynomial roots, basis, Galois group
\end{keywords}

\begin{AMS}
03G10, 11D61, 11Y50, 12E30, 12F05, 12F10, 12Y05, 20B05
\end{AMS}

\section{Introduction}
For any vector $\alpha=(\alpha_1,\ldots,\alpha_n)^T\in(\overline{\mathbb{Q}}^*)^n$ of non-zero algebraic numbers,  \emph{the exponent lattice} of $\alpha$ refers to the set of integer vectors \[\mathcal{R}_\alpha=\big\{v\in\mathbb{Z}^n\;|\;\alpha_1^{v(1)}\cdots\alpha_{n}^{v(n)}=1\big\}.\] Any vector in $\mathcal{R}_\alpha$ is called \emph{a multiplicative relation} between those $\alpha_i$, $i=1,\ldots,n$. 
For a univariate polynomial $f\in\mathbb{Q}[x]$, if $f(0)\neq0$ and $\beta_1,\beta_2,\ldots,\beta_n$ are all the complex roots of $f$ listed with multiplicity, then the exponent lattice $\mathcal{R}_\beta$ of the roots $\beta=(\beta_1,\ldots,\beta_n)^T$ is denoted simply by $\mathcal{R}_f$. For convenience we also define $\mathcal{R}_\alpha^{\mathbb{Q}}=\big\{v\in\mathbb{Z}^n\;|\;\alpha_1^{v(1)}\cdots\alpha_{n}^{v(n)}\in\mathbb{Q}\big\}$ and $\mathcal{R}_f^{\mathbb{Q}}=\mathcal{R}_\beta^{\mathbb{Q}}$. Moreover, we denote the Galois group of $f$ by $G_f=\text{Gal}(\mathcal{F}_f/\mathbb{Q})$, where $\mathcal{F}_f$ is the splitting field of $f$ over $\mathbb{Q}$. A root $\beta_i$ of $f$ is called \emph{a root of rational} if there is a positive integer $k$ so that $\beta_i^k\in\mathbb{Q}$ 
while a lattice $\mathcal{L}\subset\mathbb{Z}^n$ is called \emph{trivial} if every vector $v\in\mathcal{L}$ satisfies $v(1)=v(2)=\cdots=v(n)$.

Multiplicative relations between polynomial roots have intrigued many researchers. There are two problems lying in the core of the study of multiplicative relations between polynomial roots: (i) Is there a sufficient and necessary condition, which is easy to check, for the lattice $\mathcal{R}_f$ to be trivial?  Are there some sufficient conditions implying that $\mathcal{R}_f$ is trivial? (ii) How to develop an algorithm computing the lattice $\mathcal{R}_f$ for a given $f\in\mathbb{Q}[x]$ with $f(0)\neq0$?

To the best of our knowledge, no sufficient and necessary condition has been given for the lattice $\mathcal{R}_f$ to be trivial. However, there are many sufficient conditions in the literature. We suppose in the rest of this paragraph that $f$ is an irreducible polynomial in $\mathbb{Q}[x]$ with no root being a root of rational, whenever it is mentioned. Then by \citep[Lemma 1]{sn}, the condition  ``$G_f$ is isomorphic to the symmetric group of order $\deg(f)$'' implies that $\mathcal{R}_f$ is trivial, while by either \citep[Theorem 3]{2tran} or \citep[Theorem 1]{dixon}, $\mathcal{R}_f$ is trivial if $G_f$ is $2$-transitive on the set of the roots of $f$. Another condition implying that $\mathcal{R}_{f}$ is trivial is given by \citep[Theorem 1]{pcase}, which requires that $\deg(f)=p$ is an odd prime and for all $b,c\in\mathbb{Q}^*$, $f\neq bx^p-c$.

There are also quite a few polynomials $f\in\mathbb{Q}[x]$ with$f(0)\neq0$, such that $\mathcal{R}_f$ is non-trivial. For these polynomials an algorithm is desired to compute the lattice $\mathcal{R}_f$. Setting $U\subset\overline{\mathbb{Q}}^*$ to be the set of the roots of unity and $\beta=(\beta_1,\ldots,\beta_n)^T$ the roots of $f$ listed with multiplicity, we define \emph{the saturated lattice} of $\mathcal{R}_f$ by
\[
\begin{array}{*{20}{lll}}
\mathcal{R}_f^U&=&\{v\in\mathbb{Z}^n\;|\;\beta_1^{v(1)}\cdots\beta_n^{v(n)}\in U\}\\
&=&\{v\in\mathbb{Z}^n\;|\;\exists \lambda \in\mathbb{Z}^*, \lambda v\in\mathcal{R}_f\}.
\end{array}
\]
Then its \emph{lattice ideal} in the ring $\overline{\mathbb{Q}}[x_1,\ldots,x_n]=\overline{\mathbb{Q}}[X]$ refers to the ideal
\[
\begin{array}{*{20}{lll}}
I(\mathcal{R}_f^U)&=&\big\langle\{X^{v_+}-\beta^vX^{v_-}\;|\;v\in\mathcal{R}_f^U\}\big\rangle,
\end{array}
\]
where $v_+$ is in $\mathbb{Z}^n$ with $v_+(i)=\max\{v(i),0\}$ for $i=1,2,\dots,n$, $v_-=v_+-v$, $X^v=x_1^{v(1)}\cdots x_n^{v(n)}$ and $\beta^v=\beta_1^{v(1)}\cdots \beta_n^{v(n)}$ for any $v\in\mathbb{Z}^n$. An algorithm computing the ideal $I(\mathcal{R}_f^U)$ (which is closely related to the lattice $\mathcal{R}_f$) for a given irreducible polynomial $f\in\mathbb{Q}[x]$ is claimed to exist in Theorem 8 of \cite{structure}. This theorem is based on Theorem 6 of \cite{structure}, which indicates that if $f\in\mathbb{Q}[x]$ is irreducible and $B_f$ is the reduced Gr$\ddot{\text{o}}$bner basis of $I(\mathcal{R}_f^U)$ with respect to a lexicographic monomial order, then for any binomial in $B_f$, either it has a constant term or its terms share the same number of variables. Unfortunately, the following example due to A. Schinzel \cite[p.$\;$1]{pcase} suggests that Theorem 6 of \cite{structure} may not hold for all the irreducible polynomials.
\begin{example}\label{Schinzel}
Set
$f=x^6-2x^4-6x^3-2x^2+1$ with roots $\beta_1=0.44576\cdots$, $\beta_2=2.24333\cdots$ and
\[
\begin{array}{*{20}{llllllllll}}
\beta_3&=&-&(0.92999\cdots)&-&(1.17407\cdots)\sqrt{-1},&&\beta_4&=&\bar{\beta}_3,\\
\beta_5&=&-&(0.41455\cdots)&-&(0.52336\cdots)\sqrt{-1},&&\beta_6&=&\bar{\beta}_5.
\end{array}
\]
Computation with Mathematica shows $f$ is irreducible in $\mathbb{Q}[x]$. 
A  basis of $\mathcal{R}_f$ given by the algorithm {\tt FindRelations} described in \cite{ge} and \cite[$\S\,7.3$]{Kauers} is as follows
\begin{equation}\label{basis}
\big\{
(0, 0, -1, 0, 0, -1)^T, (-1, 0, -1, -1, 0, 0)^T, (0, 0, 0, 1, 1,
  0)^T, (-1, -1, 0, 0, 0, 0)^T
\big\}.
\end{equation}
From the basis (\ref{basis}) of $\mathcal{R}_f$ and the definition of $\mathcal{R}_f^U$, one obtains easily that $\mathcal{R}_f^U=\mathcal{R}_f$. Then
\[
\begin{array}{*{20}{lll}}
I(\mathcal{R}_f^U)&=&\big\langle\{X^{v_+}-\beta^vX^{v_-}\;|\;v\in\mathcal{R}_f^U\}\big\rangle\\
&=&\big\langle\{X^{v_+}-X^{v_-}\;|\;v\in\mathcal{R}_f\}\big\rangle.
\end{array}
\]
By the methods in \cite{markov}, the reduced Gr$\ddot{\text{o}}$bner basis of $I(\mathcal{R}_f^U)$ with respect to the lexicographic monomial order with $x_1\prec\cdots\prec x_6$ is
\[
B_f=\{-1 + x_1 x_2, -x_2 + x_3 x_4, -x_1 x_3 + x_5, -x_1 x_4 + x_6\}.
\]
The last three binomials in $B_f$ do not have the good property claimed by Theorem 6 of \cite{structure}: each of them contains no constant terms but consists of two terms that do not share a same number of variables.
\end{example}

To the best of our knowledge, there are no algorithms designed particularly to compute the lattice $\mathcal{R}_f$ for $f\in\mathbb{Q}[x]$, except for the relevant one mentioned in Theorem 8 of \cite{structure} above. But many problems in other areas can be reduced to the problem of computing a basis of the lattice $\mathcal{R}_f$, or contain it as a subproblem. For instance, based on
computing the exponent lattice of the eigenvalues of an arbitrary invertible matrix, an algorithm was proposed to compute
the Zariski closure of a finitely generated group of invertible matrices in \cite{zariski}. Additionally, an algorithm containing a subroutine which computes the exponent lattice of the roots of the  characteristic polynomial of a linear recurrence sequence was provided to compute the ideal of algebraic relations among C-finite sequences in \cite{cfinite}. What's more, a
class of loop invariants called \emph{L-invariants} introduced in \cite{polyinv} for linear loops are closely related to the exponent lattice of polynomial roots: each non-zero vector in the lattice corresponds to an L-invariant. Interestingly, a part of the invariant ideal of the linear loop is exactly the lattice ideal defined by the exponent lattice of those polynomial roots.

The algorithms proposed in \cite{ge, Kauers} (named {\tt FindRelations}) and \cite{issac} (named {\tt GetBasis}) are designed to compute the exponent lattice of  arbitrarily given non-zero algebraic numbers. Neither of them takes into account the case where the input algebraic numbers are exactly all the roots of a polynomial $f\in\mathbb{Q}[x]$ (which can be called \emph{the Galois case}). We will see in $\S\,$\ref{subsection:bottleneck} that, for the Galois case, these two algorithms become less efficient when the degree of $f$ becomes slightly larger. This indicates the necessity of developing a particular algorithm for the Galois case.
%
%

One of the main results in this paper is the assertion that if $f$ is irreducible with no root being a root of rational and $G_f$ is $2$-homogeneous, then $\mathcal{R}_f$ is trivial (Theorem \ref{thm:main}). This is proved in Section \ref{sec:main} by taking advantage of the properties of the permutation groups that are $2$-homogeneous but not $2$-transitive. This result generalizes the conditions given in \cite {sn}, \cite{2tran} and \cite{dixon} mentioned before, since both the symmetric group of order ${\deg(f)}$ (when $\deg(f)\geq2$) and any $2$-transitive group are $2$-homogeneous. According to Theorem \ref{thm:main}, we design Algorithm \ref{fastbasis} to compute the lattice $\mathcal{R}_f$ for any $f$ in a subset $E\subset\mathbb{Q}[x]$ (defined in Definition \ref{def:easy}). To be precise, for any input $f\in\mathbb{Q}[x]$, if $f\in E$ then Algorithm \ref{fastbasis} returns a basis of $\mathcal{R}_f$, otherwise it returns a symbol ``\textcolor{red}{F}''. For this, an efficiently checkable criterion (Corollary \ref{dis}) is proposed in Section \ref{deciding} to decide whether a given polynomial is in the set $E$. In Section \ref{E}, we prove that $E$ is a generic subset of $\mathbb{Q}[x]$ (Proposition \ref{truncation}). This indicates that the polynomials in the set $E$, which can be handled by Algorithm \ref{fastbasis}, constitute a large proportion of  the polynomials in $\mathbb{Q}[x]$. This theoretical result is consistent with the numerical results shown in Table \ref{FBresults} in Section \ref{algexp}. From the numerical results we can also see that Algorithm \ref{fastbasis} is effective and efficient for the polynomials in the generic subset $E$ of $\mathbb{Q}[x]$ and a large number of polynomials which were intractable before can be dealt with.

The next section is devoted to studying the rank of a special kind of matrices, which lies in the core of the proof of the main theorem (Theorem \ref{thm:main}).
\section{The Rank of  A Fractal Circulant Matrix}\label{fcm}
Suppose that $m\geq2$ is an integer. Set $b_i$, $i\in\mathbb{Z}/m\mathbb{Z}$, to be some complex numbers (resp., some complex matrices of the same dimension), then the circulant matrix (resp., the block-circulant matrix) generated by $b_i$ is given by
\[
C(b_0,\ldots,b_{m-1})=\big(b_{j-i}\big)_{i,j\in\mathbb{Z}/m\mathbb{Z}}=
\left({\begin{array}{*{20}{lll}}
b_0&b_1&\cdots&b_{m-1}\\
b_{m-1}&b_0&\cdots&b_{m-2}\\
\vdots&\vdots&\ddots&\vdots\\
b_1&b_2&\cdots&b_0
\end{array}} \right).
\]
\begin{definition}
Suppose that $m,d$ are positive integers and that $a_v$, with $v\in(\mathbb{Z}/m\mathbb{Z})^d$, are $m^d$ complex numbers.
\emph{A fractal circulant matrix} $M\in\mathbb{C}^{m^d\times m^d}$ of {\emph{order}} $m$ and \emph{depth} $d$ generated by the numbers $a_v$ is defined by
\begin{equation}\label{entry}
M=\big(a_{v-u}\big)_{u,v\in(\mathbb{Z}/m\mathbb{Z})^d}.
\end{equation}
\end{definition}
We define an order $\lessdot$ on the set $\mathbb{Z}/m\mathbb{Z}$ such that $0\lessdot1\lessdot\;\cdots\;\lessdot m-1$. Also, a lexicographic order $\prec$ is defined on the set $(\mathbb{Z}/m\mathbb{Z})^d$ as follows:  for any $v,u\in(\mathbb{Z}/m\mathbb{Z})^d$, we say $v\prec u$ iff the minimal $i\in\{1,\ldots,d\}$ such that $v(i)\neq u(i)$ satisfies $v(i)\lessdot u(i)$. In the following of this paper, the rows and columns of any fractal circulant matrix $M$ will be arranged increasingly in accordance with the order $\prec$.

\begin{example}
Set $m=3$, $d=2$ and $(a_{00},a_{01},a_{02},a_{10},a_{11},a_{12},a_{20},a_{21},a_{22})$\,$=$\,$(4, -1,2,1,2,$\\$3,-3,0,5)$, then
the fractal circulant matrix we obtain here is
\[
M=\left({\begin{array}{*{20}{rrr;{2pt/2pt}rrr;{2pt/2pt}rrr}}
4&-1&2&\textcolor{blue}{1}&\textcolor{blue}{2}&\textcolor{blue}{3}&\textcolor{red}{-3}&\textcolor{red}{0}&\textcolor{red}{5}\\
2&4&-1&\textcolor{blue}{3}&\textcolor{blue}{1}&\textcolor{blue}{2}&\textcolor{red}{5}&\textcolor{red}{-3}&\textcolor{red}{0}\\
-1&2&4&\textcolor{blue}{2}&\textcolor{blue}{3}&\textcolor{blue}{1}&\textcolor{red}{0}&\textcolor{red}{5}&\textcolor{red}{-3}\\
\hdashline[2pt/2pt]
\textcolor{red}{-3}&\textcolor{red}{0}&\textcolor{red}{5}&4&-1&2&\textcolor{blue}{1}&\textcolor{blue}{2}&\textcolor{blue}{3}\\
\textcolor{red}{5}&\textcolor{red}{-3}&\textcolor{red}{0}&2&4&-1&\textcolor{blue}{3}&\textcolor{blue}{1}&\textcolor{blue}{2}\\
\textcolor{red}{0}&\textcolor{red}{5}&\textcolor{red}{-3}&-1&2&4&\textcolor{blue}{2}&\textcolor{blue}{3}&\textcolor{blue}{1}\\
\hdashline[2pt/2pt]
\textcolor{blue}{1}&\textcolor{blue}{2}&\textcolor{blue}{3}&\textcolor{red}{-3}&\textcolor{red}{0}&\textcolor{red}{5}&4&-1&2\\
\textcolor{blue}{3}&\textcolor{blue}{1}&\textcolor{blue}{2}&\textcolor{red}{5}&\textcolor{red}{-3}&\textcolor{red}{0}&2&4&-1\\
\textcolor{blue}{2}&\textcolor{blue}{3}&\textcolor{blue}{1}&\textcolor{red}{0}&\textcolor{red}{5}&\textcolor{red}{-3}&-1&2&4
\end{array}}\right),
\]
with rows and columns arranged increasingly: $00\prec01\prec02\prec10\prec11\prec12\prec20\prec21\prec22.$
\end{example}

For a square matrix $A$ of dimension $n\times n$, we define $\text{corank}(A)=n-\text{rank}(A)$. The following theorem characterizes the corank of a fractal circulant matrix.

\begin{theorem}\label{rank}
Set $M$ to be a fractal circulant matrix of order $m$ and depth $d$, generated by some complex numbers $a_v$, $v\in(\mathbb{Z}/m\mathbb{Z})^d$. Define $\zeta_m=e^{{2\pi\sqrt{-1}}/{m}}$ and set
\[
\hat{a}_u=\sum_{v\in(\mathbb{Z}/m\mathbb{Z})^d}a_v\zeta_m^{u\cdot v}
\]
to be the discrete Fourier transformed of $a_v$ defined for all $u\in(\mathbb{Z}/m\mathbb{Z})^d$. Then
\begin{equation}\label{corank}
\text{\emph{corank}}(M)=\big|\{u\in(\mathbb{Z}/m\mathbb{Z})^d\;|\;\hat{a}_{u}=0\}\big|.
\end{equation}
\end{theorem}
\begin{proof}
The proof is inductive on the depth $d$. When $d=1$, $M$ is a circulant matrix. The conclusion follows from a theorem due to A. Schinzel in \citep[p. 5]{pcase}.

Suppose that the conclusion holds for $d=\ell$. In the following, we consider the case where $d=\ell+1$.

Define $A_{ij}\;(i,j\in\mathbb{Z}/m\mathbb{Z})$ to be the submatrix of $M$ whose rows are indexed by the set $\{v\in(\mathbb{Z}/m\mathbb{Z})^d\;|\;v(1)=i\}$ and whose columns are indexed by the set $\{v\in(\mathbb{Z}/m\mathbb{Z})^d\;|\;v(1)=j\}$. Then we have
\[
M=\left(
\begin{array}{*{20}{cccc}}
A_{00}&A_{01}&\cdots&A_{0,p-1}\\
A_{10}&A_{11}&\cdots&A_{1,p-1}\\
\vdots&\vdots&\ddots&\vdots\\
A_{p-1,0}&A_{p-1,1}&\cdots&A_{p-1,p-1}
\end{array}
\right).
\]
By comparing the entries of the matrices $A_{ij}$ and $A_{0,j-i}$ according to (\ref{entry}), one observes that
\[A_{ij}=A_{0,j-i}.\]This means that $M=C(A_{00},A_{01},\ldots,A_{0,p-1})$ is a block-circulant matrix. Noting that each $A_{0j}$ is an $m^{d-1}\times m^{d-1}$ square matrix, one observes from \citep[p.$\;$809]{bcm} that $M$ is similar to the matrix
\[\left(\begin{array}{*{20}{cccl}}
s_0&&&\\
&s_1&&\\
&&\ddots&\\
&&&s_{m-1}
\end{array}\right)
\]
where \[s_{q}=\sum_{j=0}^{m-1}\zeta_m^{qj}A_{0j}\] for each $q\in\mathbb{Z}/m\mathbb{Z}$. Since each $A_{0j}$ is a fractal circulant matrix of order $m$ and depth $d-1=\ell$, so is each $s_q$. Moreover, $A_{0j}$ is generated by the numbers $a_{{(j,v)}}$ with $v\in(\mathbb{Z}/m\mathbb{Z})^\ell$, while $s_q$ is generated by the numbers $b^{(q)}_{v}=\sum_{j=0}^{m-1}\zeta_m^{qj}a_{{(j,v)}}$. Thus by the inductive assumption that the theorem holds for depth $\ell$, we have
\[
\text{corank}(s_q)=\big|\{u\in(\mathbb{Z}/m\mathbb{Z})^\ell\;|\sum_{v\in(\mathbb{Z}/m\mathbb{Z})^\ell}b^{(q)}_{v}\zeta_m^{u\cdot v}=0\}\big|\quad\quad\quad
\]
\[
\quad\quad\quad\quad\quad\quad=\big|\{u\in(\mathbb{Z}/m\mathbb{Z})^\ell\;|\sum_{(j,v)\in(\mathbb{Z}/m\mathbb{Z})^{1+\ell}}a_{(j,v)}\zeta_m^{qj+u\cdot v}=0\}\big|
\]
\[
\;=\big|\{u\in(\mathbb{Z}/m\mathbb{Z})^\ell\;|\hat{a}_{(q,u)}=0\}\big|.\quad\quad\quad
\]
Thus \[\quad\quad\quad\text{corank}(M)=\sum_{q=0}^{m-1}\text{corank}(s_q)=\big|\{(q,u)\in(\mathbb{Z}/m\mathbb{Z})^{1+\ell}\;|\hat{a}_{(q,u)}=0\}\big|.\;\,\]
Hence the theorem holds for $d=\ell+1$ and we are done.
\end{proof}
\begin{definition}
For a fractal circulant matrix $M$ of prime order $p$ and depth $d$ generated by some rational numbers $a_v,v\in\mathbb{F}_p^d$, \emph{a slicing vector} $u\in \mathbb{F}_p^d\backslash\{\mathbf{0}\}$ of $M$ is one such that
\[
\sum_{v\in P_0} a_v=\sum_{v\in P_1} a_v=\cdots=\sum_{v\in P_{p-1}} a_v,
\]where each
\begin{equation}\label{planes}
P_j = \big\{v\in\mathbb{F}_p^d\;|\;u\cdot v=j\big\},j\in\mathbb{F}_p.
\end{equation}
We denote by $V(M)$ the set of all slicing vectors of $M$ and define the projective equivalence relation $\sim$ on the set $\mathbb{F}_p^d\backslash\{\mathbf{0}\}$ such that for $v_1,v_2\in \mathbb{F}_p^d\backslash\{\mathbf{0}\}$, $v_1\sim v_2$ iff $v_1= jv_2$ for some $j\in\mathbb{F}_p^*$. Then the number $|V(M)/$$\sim$$|$ is defined to be \emph{the slicing number} of the fractal circulant matrix $M$.
\end{definition}
\begin{lemma}\label{equcon}
Suppose that $M$ is a fractal circulant matrix of prime order $p$ and depth $d$ generated by some rational numbers $a_v,v\in\mathbb{F}_p^d$. Then, for any $u\in\mathbb{F}_p^d\backslash\{\mathbf{0}\}$, $\hat{a}_u=0$ iff $u$ is a slicing vector of $M$.
\end{lemma}
\begin{proof}
``If'': Suppose that $u$ is a slicing vector of $M$. Since $u\neq\mathbf{0}$, each $P_j\neq\emptyset$ in (\ref{planes}). Thus
\[
\begin{array}{*{20}{rcl}}
\vspace{1mm}
\hat{a}_{u}&=&\sum_{v\in\mathbb{F}_p^d}a_v\zeta_p^{u\cdot v}\\\vspace{2mm}
&=&\sum_{j=0}^{p-1}\sum_{v\in P_j}a_{v}\zeta_p^{u\cdot v}\\\vspace{2mm}
&=&\sum_{j=0}^{p-1}\sum_{v\in P_j}a_{v}\zeta_p^j\\\vspace{2mm}
&=&\big(\sum_{v\in P_j}a_{v}\big)\big(\sum_{j=0}^{p-1}\zeta_p^j\big)\\\vspace{2mm}
&=&0.
\end{array}
\]

``Only If'': Set $u\neq\mathbf{0}$ to be a vector such that $\hat{a}_{u}=0$. We define $P_j$ as in (\ref{planes}) and denote
\[
W_j=\sum_{v\in P_j}a_{v}.
\]
Then we have
\[
\begin{array}{*{20}{rcl}}\vspace{2mm}
0&=&\sum_{v\in\mathbb{F}_p^d}a_{v}\zeta_p^{u\cdot v}\\\vspace{2mm}
&=&\sum_{j=0}^{p-1}\sum_{v\in P_j}a_{v}\zeta_p^{u\cdot v}\\\vspace{2mm}
&=&\sum_{j=0}^{p-1}\sum_{v\in P_j}a_{v}\zeta_p^j\\\vspace{2mm}
&=&\sum_{j=0}^{p-1}W_j\zeta_p^j.
\end{array}
\]
This implies that the polynomial $\sum_{j=0}^{p-1}W_jx^j\in\mathbb{Q}[x]$, which is of degree at most $p-1$, vanishes at the point $x=\zeta_p$. On the other hand, the minimal polynomial of $\zeta_p$ over the field $\mathbb{Q}$ is $1+x+\cdots+ x^{p-1}$. Hence one concludes that $W_0=W_1=\cdots=W_{p-1}$, which indicates that the vector $v$ is a slicing vector of $M$.
\end{proof}
\begin{theorem}\label{whatrank}
Set $M$ to be a fractal circulant matrix of prime order $p$ and depth $d$ generated by some rational numbers $a_{v},v\in\mathbb{F}_p^d$. Suppose that \[W=\sum_{v\in\mathbb{F}_p^d}a_{v},\] then the slicing number $s(M)$ of $M$ satisfies $0\leq s(M)\leq\frac{p^d-1}{p-1}$. Moreover,
\begin{equation}\label{slicing}
\text{\emph{corank}}(M)=
\left\{
\begin{array}{*{20}{l}}
(p-1)s(M), &\text{ if }W\neq0,\\
(p-1)s(M)+1, &\text{ if }W=0.
\end{array}\right.
\end{equation}
\end{theorem}
\begin{proof}
Since $\big|\big(\mathbb{F}_p^d\backslash\{\mathbf{0}\}\big)/$$\sim$$\big|=\frac{p^d-1}{p-1}$ and $V(M)\subset\mathbb{F}_p^d\backslash\{\mathbf{0}\}$, $0\leq s(M)\leq\frac{p^d-1}{p-1}$ follows from the definition of $s(M)$. Noting that if $v_1\sim v_2$ in $\mathbb{F}_p^d\backslash\{\mathbf{0}\}$, then $v_1\in V(M)$ iff $v_2\in V(M)$ by the definition of a slicing vector. Thus $|V(M)|=(p-1)s(M)$. By Lemma \ref{equcon}, the numbers of those vectors $v\in \mathbb{F}_p^d\backslash\{\mathbf{0}\}$ such that $\hat{a}_{v}=0$ is $(p-1)s(M)$. On the other hand, $\hat{a}_{\mathbf{0}}=W$. Hence the conclusion follows from Theorem \ref{rank}.
\end{proof}
\section{The Main Theorem and Its Proof}\label{sec:main}
Suppose that $f\in\mathbb{Q}[x]$ is without multiple roots. Its Galois group $G_f$ is regarded as a permutation group operating on the set $\mathfrak{R}$ of the roots of $f$. Then $G_f$ is said to be \emph{2-homogeneous} if $\deg(f)\geq2$ and for any two subsets $\{\beta_i,\beta_j\}$ and $\{\beta_{i'},\beta_{j'}\}$ of $\mathfrak{R}$, both of cardinality two, there is an element $\sigma\in G_f$ so that
$\{\beta_{i'},\beta_{j'}\}=\big\{\sigma(\beta_i),\sigma(\beta_j)\big\}$.
\begin{definition}
If $\mathbb{F}_q$ is a Galois field with $q$ a prime power, then \em{the semi-linear group} operating on $\mathbb{F}_q$ is defined to be
\[A\varGamma L(1,q)=\big\{\eta:\mathbb{F}_q\rightarrow\mathbb{F}_q,\nu\mapsto a\nu^\sigma+b\;|\;\sigma\in \text{Aut}(\mathbb{F}_q),a\in\mathbb{F}_q^*,b\in\mathbb{F}_q\big\}.\]
\end{definition}
\begin{theorem}\label{thm:main}
Set $f(x)\in\mathbb{Q}[x]$ to be a univariate polynomial without multiple roots, so that one of its roots is not a root of rational. If the Galois group $G_f$, regarded as a permutation group operating on the set of all the complex roots of $f$, is $2$-homogeneous, then $\mathcal{R}_f^\mathbb{Q}$ is trivial.
\end{theorem}
\begin{proof}
If $G_f$ is $2$-transitive, the conclusion follows from \cite[Theorem 3]{2tran}. Thus we only need to consider the case where $G_f$ is $2$-homogeneous but not $2$-transitive.

According to \cite[Proposition 3.1]{nottrans}, $n=\deg(f)\equiv3\pmod{4}$ is a prime power and $G_f$ is similar to a subgroup $\Lambda$ of the semi-linear group $A\varGamma L(1,n)$. In other words, if we denote by $\mathfrak{R}$ the set of all the complex roots of $f$, then there is a bijection $\tau:\mathfrak{R}\rightarrow \mathbb{F}_n$ and a group isomorphism $\varrho:G_f\rightarrow \Lambda$ such that
\begin{equation}\label{commu}\sigma(r)=\big(\tau^{-1}\circ\varrho(\sigma)\circ\tau\big)(r)\end{equation}
for any $\sigma\in G_f$ and $r\in\mathfrak{R}$. This means that the group $G_f$ operates on the set $\mathfrak{R}$ in the same way as the group $\Lambda$ does on the set $\mathbb{F}_n$ once we identify the elements of these two sets through the bijection $\tau$. Hence we can index the polynomial roots $\mathfrak{R}$ in the following way: for any $\nu\in\mathbb{F}_n$, we define $r_\nu=\tau^{-1}(\nu)$, then $\mathfrak{R}=\{r_\nu\;|\;\nu\in \mathbb{F}_n\}$ and (\ref{commu}) becomes $\sigma(r_\nu)=r_{(\varrho(\sigma))(\nu)}$.

We denote by $Q$ the set of non-zero squares in the field $\mathbb{F}_n$, and by $G_f^{\nu}$ the subgroup of $G_f$ fixing $r_\nu$. Proposition 3.1 in  \cite{nottrans} also claims that the orbits of the subgroup $G_f^0$ are $\{r_0\}$, $\{r_\nu\;|\;\nu\in Q\}$ and $\{r_\nu\;|\;\nu\in -Q\}$. Setting $\eta\in\Lambda$, one concludes directly that the orbits of the subgroup $G_f^{\eta(0)}$ are $\{r_{\eta(0)}\}$, $\{r_\nu\;|\;\nu\in \eta(Q)\}$ and $\{r_\nu\;|\;\nu \in \eta(-Q)\}$.
 Moreover, $\Lambda$ contains the subgroup of all translations
 \[\Sigma=\big\{\eta:\mathbb{F}_n\rightarrow\mathbb{F}_n,\nu\mapsto \nu+b\;|\;b\in\mathbb{F}_n\big\},\]again by Proposition 3.1 in \cite{nottrans}. Hence $\Lambda$ is transitive on the set $\mathbb{F}_n$ and so is $G_f$ on the set $\mathfrak{R}$. Thus $f$ is irreducible. Since $\deg(f)\geq2$ and $f$ is irreducible, $0\not\in\mathfrak{R}$.

Suppose that
\[\prod_{\nu\in\mathbb{F}_n}r_\nu^{k_\nu}\in\mathbb{Q}\]
for some integers $k_\nu\in\mathbb{Z}$. Set $g=|G_f^0|$, then for any $\nu,\nu'\in Q$ there are $\bar{g}={g}/{|Q|}={g}/\big(\frac{1}{2}(n-1)\big)$ permutations contained in $G_f^0$ which map $r_{\nu'}$ to $r_\nu$. This also holds for any $\nu,\nu'\in-Q$. Setting $\prod_{\nu\in\mathbb{F}_n}r_\nu=b\in\mathbb{Q}$, we obtain
\[
\begin{array}{lcl}\vspace{4mm}
\mathbb{Q}&\ni&\prod_{\sigma\in G_f^0}\sigma\Big(\prod_{\nu\in\mathbb{F}_n}r_\nu^{k_\nu}\Big)\\\vspace{4mm}
&=&r_0^{gk_0}\Big(\prod_{\nu\in Q}r_{\nu}\Big)^{\bar{g}\sum_{\xi\in Q}k_\xi}\Big(\prod_{\nu\in -Q}r_{\nu}\Big)^{\bar{g}\sum_{\xi\in -Q}k_\xi}\\
&=&r_0^{gk_0}\Big(\prod_{\nu\in Q}r_{\nu}\Big)^{\bar{g}\sum_{\xi\in Q}k_\xi}\Big(\frac{b}{r_0\prod_{\nu\in Q}r_\nu}\Big)^{\bar{g}\sum_{\xi\in -Q}k_\xi}.
\end{array}
\]
Hence
\begin{equation}\label{orbit}
r_0^{gk_0-\bar{g}\sum_{\xi\in -Q}k_\xi}\Big(\prod_{\nu\in Q}r_{\nu}\Big)^{\bar{g}\big(\sum_{\xi\in Q}k_\xi-\sum_{\xi\in -Q}k_\xi\big)}\in\mathbb{Q}.
\end{equation}
Noting that for any $\eta\in\Lambda$, $\prod_{\nu\in\mathbb{F}_n}r_{\eta(\nu)}^{k_{\eta(\nu)}}=1$ and that the orbits of the group $G_f^{\eta(0)}$ are $\{r_{\eta(0)}\}$, $\{r_\nu\;|\;\nu\in\eta(Q)\}$ and $\{r_\nu\;|\;\nu\in\eta(-Q)\}$, we can similarly obtain
\begin{equation}\label{etaorbit}
r_{\eta(0)}^{gk_{\eta(0)}-\bar{g}\sum_{\xi\in -Q}k_{\eta(\xi)}}\Big(\prod_{\nu\in Q}r_{\eta(\nu)}\Big)^{\bar{g}\cdot\big(\sum_{\xi\in Q}k_{\eta(\xi)}-\sum_{\xi\in -Q}k_{\eta(\xi)}\big)}\in\mathbb{Q}
\end{equation}
with $g=|G_f^{\eta(0)}|$ and $\bar{g}=|G_f^{\eta(0)}|/|\eta(Q)|$ having respectively the same values as in (\ref{orbit}) since $G_f$ is transitive.

Set $\alpha=(\alpha_1,\alpha_2)^T=(r_0,\prod_{\nu\in Q}r_{\nu})^T$, then $\text{rank}(\mathcal{R}^{\mathbb{Q}}_\alpha)=0,1$ or $2$. In the sequel we will consider these three cases separately.
\\

\noindent{}\textcolor{blue}{\emph{Case 0: $\emph{\text{rank}}(\mathcal{R}^{\mathbb{Q}}_\alpha)=0.$}}

Let $\sigma=\varrho^{-1}(\eta)\in G_f$, then for the vector $\sigma(\alpha)=(r_{\eta(0)},\prod_{\nu\in Q}r_{\eta(\nu)})^T$, the lattice $\mathcal{R}^{\mathbb{Q}}_{\sigma(\alpha)}$ is also of rank $0$. Thus we obtain from (\ref{etaorbit}) that
\[
\left\{{\begin{array}{*{20}{rcl}}\vspace{3mm}
gk_{\eta(0)}-\bar{g}\sum_{\xi\in-Q}k_{\eta(\xi)}&=&0,\\
\bar{g}\cdot\big(\sum_{\xi\in Q}k_{\eta(\xi)}-\sum_{\xi\in-Q}k_{\eta(\xi)}\big)&=&0.
\end{array}}\right.
\]
Hence $k_{\eta(0)}=\frac{1}{\frac{1}{2}(n-1)}\sum_{\xi\in Q}k_{\eta(\xi)}=\frac{1}{\frac{1}{2}(n-1)}\sum_{\xi\in-Q}k_{\eta(\xi)}$. It follows immediately that $k_{\eta(0)}=\big(\sum_{\nu\in\mathbb{F}_n}k_\nu\big)/{n}$, which does not depend on $\eta$ at all. Since $\Lambda$ operates transitively on the set $\mathbb{F}_n$, we conclude that all $k_\nu$ share the same value.
\\

\noindent{}\textcolor{blue}{\emph{Case 2: $\emph{\text{rank}}(\mathcal{R}^{\mathbb{Q}}_\alpha)=2.$}}

In this case, we can chose two basis vectors $v,u\in\mathbb{Z}^2$ of  $\mathcal{R}^{\mathbb{Q}}_\alpha$ so that the $2\times2$ matrix $(v, u)$ is in Hermite normal form and $v(2)=0,v(1)\geq1$.  Then $r_0^{v(1)}=\alpha_1^{v(1)}\alpha_2^{v(2)}\in\mathbb{Q}$, thus $r_0$ is a root of rational. Since $f$ is irreducible, $f(x)|x^{v(1)}-r_0^{v(1)}$ in $\mathbb{Q}[x]$. This implies that every root of $f$ is a root of rational, which contradicts the assumption of the theorem.
\\

\noindent{}\textcolor{blue}{\emph{Case 1: $\emph{\text{rank}}(\mathcal{R}^{\mathbb{Q}}_\alpha)=1.$}}

Suppose that $\ell=(\ell_1,\ell_2)^T$ is a basis of $\mathcal{R}_\alpha^\mathbb{Q}$. As in Case 2, $\ell_2=0$ implies that every root of $f$ is a root of rational. Thus we may assume that $\ell_2>0$. Note that $\ell$ is also a basis of $\mathcal{R}^{\mathbb{Q}}_{\sigma(\alpha)}$. Combining this with (\ref{etaorbit}) we have
\[
\ell_2\cdot\big(gk_{\eta(0)}-\bar{g}\sum_{\xi\in-Q}k_{\eta(\xi)}\big)=\ell_1\cdot \bar{g}\cdot\big(\sum_{\xi\in Q}k_{\eta(\xi)}-\sum_{\xi\in-Q}k_{\eta(\xi)}\big),
\]which is equivalent to
\begin{equation}\label{lambdarow}
-k_{\eta(0)}+\frac{\lambda}{\frac{1}{2}(n-1)}\sum_{\xi\in Q}k_{\eta(\xi)}+\frac{1-\lambda}{\frac{1}{2}(n-1)}\sum_{\xi\in -Q}k_{\eta(\xi)}=0
\end{equation}
if we set $\lambda=\ell_1/\ell_2$. Denote $c_0=-1$, $c_\nu=\frac{\lambda}{\frac{1}{2}(n-1)}$ for $\nu\in Q$ and $c_\nu=\frac{1-\lambda}{\frac{1}{2}(n-1)}$ for $\nu\in-Q$. Then (\ref{lambdarow}) becomes
\[
\sum_{\nu\in\mathbb{F}_n}c_\nu k_{\eta(\nu)}=0,
\]
which can be re-formulated in the following way
\begin{equation}\label{etainv}
\sum_{\nu\in\mathbb{F}_n}c_{\eta^{-1}(\nu)} k_{\nu}=0.
\end{equation}

Considering the following equations with unknown integers $z_\nu$:
\begin{equation}\label{equ}
\sum_{\nu\in\mathbb{F}_n}c_{\eta^{-1}(\nu)} z_{\nu}=0,
\end{equation}
one observes from (\ref{lambdarow}) and (\ref{etainv}) that $z_\nu=1$ (for all $\nu$) is a solution. It is sufficient to prove that the coefficient matrix $A=\big(c_{\eta^{-1}(\nu)}\big)_{\small\substack{\eta\in\Lambda\;\;\\
\nu\in\mathbb{F}_n}}$ is of rank $ n-1$ for any rational number $\lambda$ (not only for $\lambda=\ell_1/\ell_2$). Since if this is true, then (\ref{etainv}) will imply that the vector $(k_\nu)_{\nu\in\mathbb{F}_n}=m\cdot(1,1,\ldots,1)^T$ for some integer $m$. Thus all $k_\nu$ share the same value.

Note that $\text{rank}(A)\leq n-1$ and that the group of all translations
\[\Sigma=\big\{\eta:\mathbb{F}_n\rightarrow\mathbb{F}_n,\nu\mapsto\nu+b\;|\;b\in\mathbb{F}_n\big\}\]
is a subgroup of $\Lambda$. We only need to proof that the submatrix $M=\big(c_{\eta^{-1}(\nu)}\big)_{\small\substack{\eta\in\Sigma\;\;\\
\nu\in\mathbb{F}_n}}$ of $A$ is of rank $n-1$. Suppose that $n=p^d$ for a prime number $p$ and a positive integer $d$, then $\mathbb{F}_n$ is a $d$-dimensional $\mathbb{F}_p$-vector space. Fixing any basis $\{\mu_1,\ldots,\mu_d\}$ of $\mathbb{F}_n$, we can define a linear isomorphism $\varphi: \mathbb{F}_n\rightarrow\mathbb{F}_p^d,\nu\mapsto v$ such that $\nu=v(1)\mu_1+\cdots+v(d)\mu_d$ and a group isomorphism $\psi: \Sigma\rightarrow\mathbb{F}_p^d,\eta\mapsto u$ such that $\eta(0)=u(1)\mu_1+\cdots+u(d)\mu_d$. One notices that $\varphi\big(\eta(0)\big)=u=\psi(\eta)$ and $\varphi\big(\eta^{-1}(\nu)\big)=\varphi\big(\nu-\eta(0)\big)=\varphi(\nu)-\varphi\big(\eta(0)\big)=v-u$. Then $\eta^{-1}(\nu)=\varphi^{-1}(v-u)$ and $M=\big(c_{\eta^{-1}(\nu)}\big)_{\small\substack{\eta\in\Sigma\;\;\\
\nu\in\mathbb{F}_n}}=\big(c_{\varphi^{-1}(v-u)}\big)_{\small u,v\in\mathbb{F}_p^d}$ is a fractal circulant matrix of prime order $p$ and depth $d$ generated by the rational numbers $c_{\varphi^{-1}(v)},v\in\mathbb{F}_p^d$.

In the following we prove that $M$ has no slicing vectors by contradiction.

Suppose that $\tilde{u}\in\mathbb{F}_p^d$ is a slicing vector of $M$. As usual, we denote $P_j=\{v\in\mathbb{F}_p^d\;|\;\tilde{u}\cdot v=j\}$, $j\in\mathbb{F}_p$. Then
\[
\sum_{v\in P_0}c_{\varphi^{-1}(v)}=\sum_{v\in P_1}c_{\varphi^{-1}(v)}=\cdots=\sum_{v\in P_{p-1}}c_{\varphi^{-1}(v)}=\frac{1}{p}\sum_{v\in\mathbb{F}_p^d}c_{\varphi^{-1}(v)}=\frac{1}{p}\sum_{\nu\in\mathbb{F}_n}c_{\nu}=0.
\]
Noting that $P_{p-1}=-P_1$, one concludes that
\begin{equation}\label{neg}
\varphi^{-1}(P_{p-1})=-\varphi^{-1}(P_1).
\end{equation} Since $\varphi(0)=\mathbf{0}\in P_0$, $0\in\varphi^{-1}(P_0)$. Hence $0\not\in\varphi^{-1}(P_1)$ and $0\not\in\varphi^{-1}(P_{p-1})$. Denote the number of squares in the set $\varphi^{-1}(P_1)$ by $s=|Q\cap\varphi^{-1}(P_1)|$ and the number of non-squares by $t=|(-Q)\cap\varphi^{-1}(P_1)|$, then $|Q\cap\varphi^{-1}(P_{p-1})|=t$ and $|(-Q)\cap\varphi^{-1}(P_{p-1})|=s$ follow from (\ref{neg}). We obtain
\[0=\sum_{v\in P_1}c_{\varphi^{-1}(v)}=\sum_{\nu\in Q\cap\varphi^{-1}(P_1)}c_{\nu}+\sum_{\nu\in (-Q)\cap\varphi^{-1}(P_1)}c_{\nu}=\frac{s\lambda}{\frac{1}{2}(n-1)}+\frac{t(1-\lambda)}{\frac{1}{2}(n-1)},\]and
\[0=\sum_{v\in P_{p-1}}c_{\varphi^{-1}(v)}=\sum_{\nu\in Q\cap\varphi^{-1}(P_{p-1})}c_{\nu}+\sum_{\nu\in (-Q)\cap\varphi^{-1}(P_{p-1})}c_{\nu}=\frac{t\lambda}{\frac{1}{2}(n-1)}+\frac{s(1-\lambda)}{\frac{1}{2}(n-1)}.\]
Adding these two equations we obtain $s+t=0$, which contradicts the fact that $s+t=|\varphi^{-1}(P_1)|=|P_1|=n/p$. Hence $M$ has no slicing vectors. By Theorem \ref{whatrank} we claim that $\text{corank}(M)=1$, which is what we want.
\end{proof}

Once the exponent lattice $\mathcal{R}_f$ defined by the roots of a polynomial $f\in\mathbb{Q}[x]$ is known to be trivial, a basis of $\mathcal{R}_f$ can be obtained directly according to the following proposition.
\begin{proposition}\label{tribas}
If $f\in\mathbb{Q}[x]$ is monic, $f(0)\neq0$, $\mathbf{1}=(1,1,\ldots,1)^T\in\mathbb{Z}^{\deg(f)}$ and $\mathcal{R}_f$ is trivial, then exactly one of the following cases holds:\\
(i) $\mathcal{R}_f=\{\mathbf{0}\}$, if $f(0)\not\in\{1,-1\}$;\\
(ii) $\mathcal{R}_f=\{k\mathbf{1}|\;k\in\mathbb{Z}\}$, if $(-1)^{\deg(f)}\cdot f(0)=1$;\\
(iii) $\mathcal{R}_f=\{2k\mathbf{1}|\;k\in\mathbb{Z}\}$, if $(-1)^{\deg(f)}\cdot f(0)=-1$.
\end{proposition}
\begin{proof}
Suppose that $\deg(f)=n$ and $\beta_1^{k_1}\cdots\beta_n^{k_n}=1$ for the roots $\beta_i$ and some integers $k_i$. Then $k_1=\cdots=k_n=k$ for an integer $k$ since $\mathcal{R}_f$ is trival. Hence $1=(\beta_1\cdots\beta_n)^k=\big((-1)^n\cdot f(0)\big)^k$. If $f(0)\not\in\{1,-1\}$, then $k=0$. The rest two cases also follow directly.
\end{proof}
\section{Deciding $2$-Homogeneous Galois Groups}\label{deciding}
The purpose of this section is to decide whether a given polynomial satisfies the assumptions of Theorem \ref{thm:main}.

From the proof of  Theorem \ref{thm:main}, one notes that a polynomial $f$ satisfying those assumptions is necessarily irreducible since $G_f$ is transitive either when $G_f$ is $2$-transitive or when it is $2$-homogeneous but not $2$-transitive.

According to \cite[Proposition 5.2]{ror}, either all the roots of $f$ are roots of rational or none of them is a root of rational. Moreover, \cite[Algorithm 5]{ror} distinguishes these two cases for a given irreducible polynomial $f$.

The problem is reduced to deciding whether $G_f$ is $2$-homogeneous or not, provided that none of the roots of the irreducible polynomial $f$ is a root of rational.  The approach shown below needs not compute the Galois group $G_f$ at all.

For any $g\in\mathbb{Q}[x] \;(\deg(g)>1)$ with complex roots $\beta_1,\ldots,\beta_{\deg(g)}$ listed with multiplicity, define
\[g_{[2]}(x)=\prod_{1\leq i<j\leq \deg(g)}(x-\beta_i\beta_j).\]
We observe that $g_{[2]}(x)\in\mathbb{Q}[x]$ and  the following proposition holds:
\begin{proposition}\label{2hiff}
For an irreducible polynomial $f\in\mathbb{Q}[x]$ with no root being a root of rational, $G_f$ is $2$-homogeneous iff $f_{[2]}$ is irreducible in $\mathbb{Q}[x]$.
\begin{proof}
``If'': This is by \cite[Lemma 5.3]{generic} (taking $k=2$ therein). In fact we do not need the assumption that none of the roots of $f$ is a root of rational in this direction.

``Only If'': Denote $n=\deg(f)$ and set $\beta_1,\ldots,\beta_n$ to be the roots of $f$. Setting $g=f_{[2]}$, we claim that $g$ has no multiple roots. Since if $g$ has multiple roots, then $\beta_i\beta_j=\beta_{i'}\beta_{j'}$ for some $i<j$, $i'<j'$, $\{i,j\}\neq\{i',j'\}$. Now that $\big|\{i,j\}\cap\{i',j'\}\big|\leq1$, $\beta_i\beta_j\beta_{i'}^{-1}\beta_{j'}^{-1}=1$ results in a non-trivial relation in $\mathcal{R}_f$. However, $\mathcal{R}_f$ is trivial by Theorem \ref{thm:main}, which is a contradiction.

Denote by $\mathcal{F}_f$ and $\mathcal{F}_g$ the splitting fields of $f$ and $g$ over $\mathbb{Q}$ respectively. Then by Galois theory, $\varphi: G_f\rightarrow G_g$, $\sigma\mapsto\sigma|_{_{\mathcal{F}_g}}$ is a surjective group homomorphism with kernel $N=\text{Gal}(\mathcal{F}_f/\mathcal{F}_g)$. Suppose that $\beta_i\beta_j$ and $\beta_{i'}\beta_{j'}$ ($i<j,i'<j'$) are two distinct roots of $g$. Since $G_f$ is $2$-homogeneous, there is a $\sigma\in G_f$ such that $\{\sigma(\beta_i),\sigma(\beta_j)\}=\{\beta_{i'},\beta_{j'}\}$. Thus $\sigma|_{_{\mathcal{F}_g}}(\beta_i\beta_j)=\sigma(\beta_i\beta_j)=\sigma(\beta_i)\sigma(\beta_j)=\beta_{i'}\beta_{j'}$. This means $G_g$ is transitive regarded as a permutation group operating on the set of all the roots of $g$. Hence $g=f_{[2]}$ is irreducible in $\mathbb{Q}[x]$.
\end{proof}
\end{proposition}
\begin{corollary}\label{dis}
For an irreducible polynomial $f\in\mathbb{Q}[x]$ with no root being a root of rational, $G_f$ is $2$-homogeneous iff the degree of the minimal polynomial of the number $(\beta_1\beta_2)$ over $\mathbb{Q}$ equals $\deg(f)(\deg(f)-1)/2$.
\end{corollary}
\begin{proof}
It is obvious that $f_{[2]}$ is irreducible iff $\deg(\beta_1\beta_2)=\deg(f)(\deg(f)-1)/2$.
\end{proof}

Thus we can decide whether $G_f$ is $2$-homogeneous by computing the degree of the minimal polynomial of the number $(\beta_1\beta_2)$, which can be done efficiently.
\section{What Kind of and How Many Polynomials Can Be Handled}\label{E}
In this section, we define (Definition \ref{def:easy}) the set of polynomials $f$ such that we can compute the lattice $\mathcal{R}_f$ efficiently by using Theorem \ref{thm:main}, together with some results in \cite{issac}.
Moreover, we prove that this set is a generic subset of $\mathbb{Q}[x]$ (Corollary \ref{realtrunc}) in the sense that the probability measure of its truncations tends to one as the upper bound of the height of those polynomials that are considered becomes larger. In fact, we prove a more general result (Proposition \ref{truncation}) indicating that $E$ is generic in a similar sense as mentioned above, even if we only take into account those polynomials with some coefficients fixed properly.
\subsection{Defining the Set of the Polynomials that Can Be Handled}\label{Edef}
\begin{definition}\label{def:easy}
The set $E\subset\mathbb{Q}[x]$ is defined to be the set of polynomials $f$ so that both the following conditions hold:\\
$(i)$ $\exists c\in\mathbb{Q^*}$, $g\in\mathbb{Q}[x],k\in\mathbb{Z}_{\geq1}$ so that $f=cg^k$, $g$ is irreducible and $x{\not|}\,g(x)$;\\
$(ii)$ every root of $g$ is a root of rational or $g_{[2]}$ is irreducible.
\end{definition}

To understand why $\mathcal{R}_f$ can be efficiently computed for $f\in E$, one needs to observe first that $\mathcal{R}_f$ can be derived from $\mathcal{R}_g$ if $f=cg^k$. More generally, the following proposition follows from \cite[Definition 3.1]{issac} directly:

\begin{proposition}\label{multi}
 Suppose that $\ell_s{\geq0}\;(s=1,\ldots,n)$ are nonnegative integers, and that $\beta=\big(\beta_1,\beta_2,\ldots,\beta_{n+\sum_{s=1}^n\ell_s}\big)^T\in\big(\overline{\mathbb{Q}}^*\big)^{n+\sum_{s=1}^n\ell_s}$ satisfies
\[
\beta_{n+\sum_{s=1}^{i-1}{\ell_s}+1}=\beta_{n+\sum_{s=1}^{i-1}{\ell_s}+2}=\cdots=\beta_{n+\sum_{s=1}^{i-1}{\ell_s}+\ell_i}=\beta_{i}
\]
for $i=1,\ldots,n$. If $\mathcal{B}$ is a triangular basis $($as defined in  \emph{\cite[Definition 3.1]{issac}}$)$ of $\mathcal{R}_{\bar{\beta}}$, with $\bar{\beta}=(\beta_1,\ldots,\beta_n)^T$, then a triangular basis of $\mathcal{R}_\beta$ is given by:
\begin{equation}\label{ltb}
\mathcal{B}\cup\{\mathbf{\varepsilon}_{ij}\}_{1\leq i\leq n,1\leq j \leq \ell_i}.
\end{equation}
Here each $\mathbf{\varepsilon}_{ij}$ is in $\mathbb{Z}^{n+\sum_{s=1}^n\ell_s}$, whose coordinates $\varepsilon_{ij}(\iota),\iota=1,2,\ldots,n+\sum_{s=1}^n\ell_s$, are given by
\[
\varepsilon_{ij}(\iota)=\left\{{\begin{array}{*{20}{rl}}
-1,&\text{ if } \iota=i,\\
1,&\text{ if } \iota=n+\sum_{s=1}^{i-1}\ell_s+j,\\
0,&\text{ else. }
\end{array}}
\right.
\]
Each vector in $\mathcal{B}$ is in $\mathbb{Z}^n$, but in $(\ref{ltb})$ each of them is regarded as a vector in $\mathbb{Z}^{n+\sum_{s=1}^n\ell_s}$, with the extra coordinates indexed by $\iota>n$ being zeros.
\end{proposition}

\begin{example}
Set $g=x^2-3x-1$ and $f=g^2$, with roots $(\beta_1,\beta_2)^T$ and $(\beta_1,\beta_2,\beta_1,\beta_2)^T$ respectively. Then a triangular basis of $\mathcal{R}_g$ is $\big\{(2,2)^T\big\}$ while a triangular basis of $\mathcal{R}_f$ is given by $\big\{(2,2,0,0)^T,(-1,0,1,0)^T,(0,-1,0,1)^T\big\}$.
\end{example}

The problem is reduced to computing $\mathcal{R}_g$ for an irreducible  polynomial $g$ so that either of the following conditions holds: (i) all the roots of $g$ are roots of rational; (ii) none of the roots of $g$ is a root of rational and $g_{[2]}$ is irreducible.  We use the techniques developed in \cite[$\S\,2.1,\S\,2.2.1,\S\,3.2$]{issac}, which are parts of the main algorithm {\tt GetBasis} therein, to deal with the former case. From the numerical results in \cite[Table 2]{issac} we see that this case can be handled extremely efficiently. For the latter case, Theorem \ref{thm:main} and Proposition \ref{tribas} applies and $\mathcal{R}_g$ can be obtained directly after proving the irreducibility of $g_{[2]}$  by computing $\deg(\beta_1\beta_2)$ for any two roots $\beta_1,\beta_2$ of $g$. This last step can also be done efficiently by standard methods developed in the filed of computational algebraic number theory.

\subsection{The Set $E$ Is Generic}

For a polynomial $f\in\mathbb{Z}[x]$ of degree at most $n$, we denote by $c_{f,i}$ the coefficient of $f$ with respect to the term $x^i$, $i=0,1,\ldots,n$. Then the \emph{height} of $f$ is defined by \[h(f)=\max_{0\leq i\leq n}|c_{f,i}|.\] Define
$
\mathbb{Z}_{H,n}[x]=\big\{f\in\mathbb{Z}[x]\;|\;h(f)\leq H, \deg(f)\leq n\big\}
$, then the set
\[E_{H,n}=E\cap\mathbb{Z}_{H,n}[x]\]
is called a \emph{truncation} of the set $E$. Noting that $\mathbb{Z}_{H,n}[x]$ is a finite set of cardinality $(2H+1)^{n+1}$, we can equip it with the probability measure $\mathscr{P}_{H,n}$ determined by the discrete uniform distribution on it.

Suppose that $D$ is a subset of the set $\{0,1,\ldots,n\}$. If $D\neq\emptyset$, we denote by $\mathbb{Z}^D$ the set of those vectors with integer coordinates indexed by $D$. If $D=\emptyset$, we set $\mathbb{Z}^D$ to be $\{\,\vdash\}$, a set containing only a special symbol ``$\,\vdash$''. For any $v\in\mathbb{Z}^D$, we define $\mathbb{Z}_{H,n,D,v}[x]=\mathbb{Z}_{H,n}[x]$ if $D=\emptyset$ and $v=\;\vdash$, while setting
\[
\mathbb{Z}_{H,n,D,v}[x]=\big\{f\in\mathbb{Z}_{H,n}[x]\;|\;c_{f,i}=v(i),\forall i\in D\big\}
\]when $D\neq\emptyset$ and $v$ is a vector in $\mathbb{Z}^D$. The set $\mathbb{Z}_{H,n,D,v}[x]$ consists of the polynomials in $\mathbb{Z}_{H,n}[x]$ whose coefficients indexed by the set $D$ are equal to the corresponding integer coordinates of the vector $v$. Again we can equip the finite set $\mathbb{Z}_{H,n,D,v}[x]$ with the probability measure $\mathscr{P}_{H,n,D,v}$ determined by the discrete uniform distribution on it. We define $\hat{D}_{n,v}=\{0,1,\ldots,n\}$ if $D=\emptyset$ and $v=\;\vdash$, while setting \[\hat{D}_{n,v}=\big\{0,1,\ldots,n\big\}\big\backslash\big\{i\in D\;|\;v(i)=0\big\}\] when $D\neq\emptyset$ and $v$ is a vector in $\mathbb{Z}^D$. One observes that GCD$(\hat{D}_{n,v})>1$ iff $\mathbb{Z}_{H,n,D,v}[x]\subset\mathbb{Z}[x^r]$ for some integer $r>1$. Setting $E_{H,n,D,v}=E\cap\mathbb{Z}_{H,n,D,v}[x]$, we have the following result:
\begin{proposition}\label{truncation}
For any integer $n{\geq2}$, any subset $D\subset\{0,1,\ldots,n\}$ with $0\leq|D|\leq n-1$ and any $v\in\mathbb{Z}^D$ such that $\{0,n\}\subset\hat{D}_{n,v}$ and $\emph{GCD}(\hat{D}_{n,v})=1$, the following equality holds
 \[
 \lim\limits_{H\to\infty}\mathscr{P}_{H,n,D,v}(E_{H,n,D,v})=1.
 \]
\end{proposition}
\begin{proof}
Denote by $S_n$ the symmetric group of order $n$. Then, by \cite[Theorem 1]{sym} (let $K=k=\mathbb{Q}, r=1, \mathbf{t}=
\{\alpha_i\;|\;i\not\in D\}$ and $s=n+1-|D|$ therein), we know that there is a positive number $c(n)$ depending only on $n$ such that
\[
\big|\{f\in\mathbb{Z}_{H,n,D,v}[x]\;|\;G_f\not\cong S_n\}\big|\leq c(n)H^{n+\frac{1}{2}-|D|}\log H
\]
for all integers $H\geq 1$. Set $T_{H,n,D,v}=\{f\in\mathbb{Z}_{H,n,D,v}[x]\;|\;G_f\not\cong S_n\}$, then
\[
\limsup\limits_{H\to\infty}\mathscr{P}_{H,n,D,v}(T_{H,n,D,v})=\limsup\limits_{H\to\infty}\frac{|T_{H,n,D,v}|}{(2H+1)^{n+1-|D|}}\leq\limsup\limits_{H\to\infty}\frac{c(n)H^{n+\frac{1}{2}-|D|}\log H}{(2H+1)^{n+1-|D|}}=0.
\]
Define $S_{H,n,D,v}=\mathbb{Z}_{H,n,D,v}[x]{\big\backslash }T_{H,n,D,v}$, then
\[
\lim\limits_{H\to\infty}\mathscr{P}_{H,n,D,v}(S_{H,n,D,v})=1.
\]
Now it is sufficient to prove that $E_{H,n,D,v}\supset S_{H,n,D,v}$. Set $f\in S_{H,n,D,v}$,  then $G_f\cong S_n$, $\deg(f)=n$ and $f$ is irreducible. Since $n\geq2$, $x{\not|}\,f(x)$. Thus Definition \ref{def:easy} (i) is satisfied. Note that either every root of $f$ is a root of rational or none of its roots is a root of rational. In the latter case, noting that $n\geq 2$ and $G_f\cong S_n$ is $2$-homogeneous, we conclude that $f_{[2]}$ is irreducible from Proposition \ref{2hiff}. By now we have proven that $f\in E_{H,n,D,v}$. Thus $E_{H,n,D,v}\supset S_{H,n,D,v}$.
\end{proof}

When $D=\emptyset$ and $v=\;\vdash$, Proposition \ref{truncation} gives:
\begin{corollary}\label{realtrunc}
For any integer $n\geq2$, $\lim\limits_{H\to\infty}\mathscr{P}_{H,n}(E_{H,n})=1$.
\end{corollary}


Defining $\bar{\mathbb{Z}}_{H,n,D,v}[x]=\{f\in{\mathbb{Z}}_{H,n,D,v}[x]\;|\;c_{f,n}\neq0\}$ to be the set of the polynomials in ${\mathbb{Z}}_{H,n,D,v}[x]$ of degree $n$, one observes that \[S_{H,n,D,v}\subset\bar{\mathbb{Z}}_{H,n,D,v}[x].\] Setting $\bar{E}_{H,n,D,v}=E\cap\bar{\mathbb{Z}}_{H,n,D,v}[x]$ and $\bar{S}_{H,n,D,v}=\{f\in\bar{\mathbb{Z}}_{H,n,D,v}[x]\;|\;G_f\cong S_n\}$, one concludes that \[\bar{S}_{H,n,D,v}=S_{H,n,D,v}\cap\bar{\mathbb{Z}}_{H,n,D,v}[x]=S_{H,n,D,v}.\]Denoting by $\bar{\mathscr{P}}_{H,n,D,v}$ the probability measure determined by the discrete uniform distribution on the set $\bar{\mathbb{Z}}_{H,n,D,v}[x]$, we have
\[
\begin{array}{lcl}
\bar{\mathscr{P}}_{H,n,D,v}(\bar{S}_{H,n,D,v})&=&{|\bar{S}_{H,n,D,v}|}\div{\big|\bar{\mathbb{Z}}_{H,n,D,v}[x]\big|}\\
&\geq&{|{S}_{H,n,D,v}|}\div{\big|{\mathbb{Z}}_{H,n,D,v}[x]\big|}\\
&=&\mathscr{P}_{H,n,D,v}(S_{H,n,D,v}).
\end{array}
\]Then the following corollary holds:
\begin{corollary}\label{degn}
For any integer $n{\geq2}$, any subset $D\subset\{0,1,\ldots,n\}$ with $0\leq|D|\leq n-1$ and any $v\in\mathbb{Z}^D$ such that $\{0,n\}\subset\hat{D}_{n,v}$ and $\emph{GCD}(\hat{D}_{n,v})=1$, the following equality holds
 \[
 \lim\limits_{H\to\infty}\bar{\mathscr{P}}_{H,n,D,v}(\bar{E}_{H,n,D,v})=1.
 \]
\end{corollary}
\begin{proof}
As in the proof of Proposition \ref{truncation}, we have
\[
\bar{E}_{H,n,D,v}\supset\bar{S}_{H,n,D,v}\text{ and }\lim\limits_{H\to\infty}\mathscr{P}_{H,n,D,v}(S_{H,n,D,v})=1.
\]
Thus
\[
\begin{array}{lcl}
 \liminf\limits_{H\to\infty}\bar{\mathscr{P}}_{H,n,D,v}(\bar{E}_{H,n,D,v})&\geq& \liminf\limits_{H\to\infty}\bar{\mathscr{P}}_{H,n,D,v}(\bar{S}_{H,n,D,v})\\
 &\geq&\liminf\limits_{H\to\infty}\mathscr{P}_{H,n,D,v}(S_{H,n,D,v})\\
&=&1.
\end{array}
\]
\end{proof}

Setting $D=\emptyset$, $v=\;\vdash$, $\bar{\mathbb{Z}}_{H,n}[x]=\bar{\mathbb{Z}}_{H,n,\emptyset,\vdash}[x]=\{f\in{\mathbb{Z}}_{H,n}[x]\;|\;c_{f,n}\neq0\}$, $\bar{E}_{H,n}=\bar{E}_{H,n,\emptyset,\vdash}=E\cap\bar{\mathbb{Z}}_{H,n}[x]$ and $\bar{\mathscr{P}}_{H,n}=\bar{\mathscr{P}}_{H,n,\emptyset,\vdash}$, we have
\begin{corollary}\label{degnrealtrunc}
For any integer $n\geq2$, $\lim\limits_{H\to\infty}\bar{\mathscr{P}}_{H,n}(\bar{E}_{H,n})=1$.
\end{corollary}

We say the polynomials in the subset $E\subset\mathbb{Q}[x]$ are \emph{generic} in the sense that Proposition \ref{truncation} and Corollary \ref{realtrunc}--\ref{degnrealtrunc} hold.
\section{Algorithm and Numerical Results}\label{algexp}
In this section, we first summarize $\S\,\ref{Edef}$ to obtain Algorithm \ref{fastbasis} efficiently computing a basis of $\mathcal{R}_f$ for any $f\in E$. Then by showing some randomly generated examples, we point out the drawback of directly applying the state-of-the-art algorithms, which aim to compute the exponent lattice of general non-zero algebraic numbers, to the inputs which are of the Galois case. Finally, we show the superiority of Algorithm \ref{fastbasis} dealing with randomly generated polynomials by a great deal of examples.

Algorithm \ref{fastbasis} is implemented with Mathematica. The numerical results in this section are all obtained on a laptop of  WINDOWS 7 SYSTEM with 4GB RAM and a 2.53GHz Intel Core i3 processor with 4 cores.
\subsection{The Algorithm for the Roots of a Generic Polynomial}
According to $\S\,\ref{Edef}$, we design Algorithm \ref{fastbasis} to compute an exponent lattice basis of the roots of any polynomial $f\in E$. The algorithm returns an ``\textcolor{red}{F}'' if $f\not\in E$. Thus this is not a complete algorithm computing an exponent lattice basis of the roots of an arbitrary polynomial.

\begin{remark}\label{complete}
If we use either one of the algorithms {\tt FindRelations} and {\tt GetBasis} to compute $\mathcal{R}_f$ directly whenever an ``\textcolor{red}{F}'' is obtained in Algorithm \ref{fastbasis},  then it can be modified to be a complete one.
\end{remark}
\begin{algorithm}
	\caption{FastBasis}\label{fastbasis}
	\begin{algorithmic}[1]
	\REQUIRE A polynomial $f\in\mathbb{Z}[x]$.
	\ENSURE A basis of $\mathcal{R}_f$ if $f\in E$, ``\textcolor{red}{F}'' if otherwise.
	\STATE\textbf{if} \big($f$ has at least two co-prime irreducible factors or $x|f(x)$\big) \textbf{then }\{\textbf{return }\textcolor{red}{F}\}\textbf{ end if};
	\STATE Suppose that the only irreducible factor of $f$ is $g$ (i.e., $f=cg^k$, $c\in \mathbb{Q}^*$, $k\geq1$, $g\in\mathbb{Q}[x]$);
	 \IF {(all the roots of $g$ are roots of rational)}
	\STATE Compute a basis of $\mathcal{R}_g$ by \cite{issac} {\tt GetBasis} and a basis of $\mathcal{R}_f$ by Proposition \ref{multi};
	\ELSE
	\STATE Suppose that $\beta_1$ and $\beta_2$ are any two roots of $g$ and $d=\deg\big({\tt MinimalPolynomial}(\beta_1\beta_2)\big)$;
	
	\STATE\textbf{if} $(d<\deg(g)(\deg(g)-1)/2)$ \textbf{then }\{\textbf{return }\textcolor{red}{F}\}\textbf{ end if};
	\STATE Compute a basis of $\mathcal{R}_g$ from Proposition \ref{tribas} and a basis of $\mathcal{R}_f$;
	\ENDIF
	\STATE\textbf{return }the basis of $\mathcal{R}_f$
	\end{algorithmic}
\end{algorithm}
\subsection{The Bottleneck of  The Existing Algorithms}\label{subsection:bottleneck}
In this subsection, we will see that the performance of the algorithms {\tt FindRelations} and {\tt GetBasis} on randomly generated inputs, which are of the Galois case, is not very satisfactory.

By applying the algorithm {\tt FastBasis} to a randomly generated polynomial $f$ and by applying the algorithms {\tt FindRelations} and {\tt GetBasis} to its roots, we compare these three algorithms and show the results in Table \ref{compare}. The polynomials $f^{(i)}$ are randomly picked from the set $\mathbb{Z}_{10,4}[x]$ while $g^{(i)}$ and $h^{(i)}$ are from the sets $\mathbb{Z}_{10,5}[x]$ and $\mathbb{Z}_{10,9}[x]$ respectively. The acronym ``OT'' means that the algorithm does not return an answer within two hours.

\begin{table}[H]\label{compare}
    \centering
\caption{Comparing the State-of-the-Art Algorithms with {\tt FastBasis}}
\begin{tabular}{|c|c|c|c|c|c|}
 \hline
 \multirow{2}{*}{class} &\multirow{2}{*}{example} &\multicolumn{3}{c|}{runtime (s)}\\
 \cline{3-5}
 && {\tt FindRelations}&{\tt GetBasis}  & {\tt FastBasis}\\
 \hline
 \multirow{3}{*}{\shortstack{$n=4$\\\\$H=10$}} & $f^{(1)}$  &52.4249&133.851 & 0.00664\\
 \cline{2-5}
&$f^{(2)}$ & 34.0127 &90.3257&0.00652 \\
   \cline{2-5}
  & $f^{(3)}$& 54.9171&144.214 & 0.00746\\
   \hline
 \multirow{3}{*}{\shortstack{$n=5$\\\\$H=10$}} & $g^{(1)}$ & OT&OT & 0.01056\\
 \cline{2-5}
& $g^{(2)}$ & OT &OT& 0.01009\\
   \cline{2-5}
  &$g^{(3)}$ & OT &OT& 0.00901\\
   \hline
 \multirow{3}{*}{\shortstack{$n=9$\\\\$H=10$}} & $h^{(1)}$ & OT&OT & 0.04895\\
 \cline{2-5}
& $h^{(2)}$ & OT &OT& 0.04500\\
   \cline{2-5}
  & $h^{(3)}$ & OT &OT& 0.05065\\
 \hline
 \end{tabular}
 \end{table}
From the table we see that both algorithms ({\tt FindRelations} and {\tt GetBasis}) become less efficient when the degree bound $n$ becomes slightly larger. In contrast, for every example in the table, the algorithm {\tt FastBasis} returns an answer successfully in a short time. This suggests that the special techniques developed for the Galois case are effective and promising.

In the next subsection, we will see that the algorithm {\tt FastBasis} does not always return an answer successfully. Fortunately, it does success most of the time, as indicated by Corollary \ref{realtrunc} . Moreover, the average runtime of those success examples in each class is, to some extent, satisfactory.
\subsection{The Superiority of the New Approach}\label{sup}
We test  the algorithm {\tt FastBasis} by a large number of randomly generated polynomials from different classes of the form $\mathbb{Z}_{H,n}[x]$. The results are shown in Table \ref{FBresults}.
 \newpage
The notation ``$\#$example'' denotes the number of the examples that are generated in a single class, while ``$\#$success'' denotes the number of those examples among them for which the algorithm returns a lattice basis successfully within two hours. The notation ``$\#$F'' stands for the number of the generated polynomials in a class that are proved to be outside the set $E$ within two hours, while the  notation ``$\#$OT'' shows the number of examples for which the algorithm returns no results within two hours.
 \begin{table}[H]\label{FBresults}
    \centering
\caption{Testing {\tt FastBasis} by Random Polynomials}
\begin{tabular}{|c|c||c|c|c|c||c|c|}
\hline
\multicolumn{2}{|c||}{\multirow{2}{*}{class}} &\multirow{2}{*}{\#example}&\multirow{2}{*}{\#success}&\multirow{2}{*}{\#OT}&\multirow{2}{*}{\#F}& \multirow{2}{*}{\shortstack{average runtime (s) \\for success examples}}\\
\multicolumn{2}{|c||}{}&&&&&\\
\hline
 \multirow{3}{*}{$n=6$} & $H=10$  &10000&8941&0&1059&0.0114477\\
 \cline{2-7}
&$H=20$ & 10000&9470&0&530&0.0127577\\
   \cline{2-7}
  & $H=50$&10000&9785&0&215&0.0139599\\
  \hline
   \hline
   \multirow{3}{*}{$n=8$} & $H=10$  &10000&9045&0&955&0.0287412\\
 \cline{2-7}
&$H=20$ & 10000 &9557&0&443&0.0326949\\
   \cline{2-7}
  & $H=50$& 10000&9816&0&184&0.0376648\\
  \hline
   \hline
   \multirow{3}{*}{$n=9$} & $H=10$  &10000&9079&0&921&0.0454840\\
 \cline{2-7}
&$H=20$ &10000&9540&0&460&0.0521535\\
   \cline{2-7}
  & $H=50$&10000&9814&0&186&0.0599880\\
  \hline
   \hline
   \multirow{3}{*}{$n=10$} & $H=10$  &10000&9173&0&827&0.0729700\\
 \cline{2-7}
&$H=20$ & 10000 &9542&0&458&0.0829247\\
   \cline{2-7}
  & $H=50$& 10000&9817&0&183&0.0963495\\
  \hline
   \hline
     \multirow{2}{*}{$n=15$} & $H=10$  &10000&9175&0&825&0.4842220\\
 \cline{2-7}
&$H=50$ & 10000 &9827&0&173&0.6751220\\
  \hline
   \hline
     \multirow{2}{*}{$n=20$} & $H=10$  &10000&9213&0&787&2.44977\\
 \cline{2-7}
&$H=50$ & 10000 &9831&0&169&3.24327\\
  \hline
   \hline
     \multirow{2}{*}{$n=30$} & $H=10$  &10000&9321&0&679&22.2455\\
 \cline{2-7}
&$H=50$ & 10000 &9850&0&150&29.0741\\
  \hline
   \hline
       \multirow{2}{*}{$n=40$} & $H=10$  &100&96&0&4&127.722\\
 \cline{2-7}
&$H=50$ & 100 &100&0&0&157.399\\
  \hline
   \hline
       \multirow{2}{*}{$n=50$} & $H=10$  &100&95&0&5&579.485\\
 \cline{2-7}
&$H=50$ & 100 &99&0&1&673.213\\
  \hline
   \hline
       \multirow{2}{*}{$n=60$} & $H=10$  &35&32&0&3&2421.11\\
 \cline{2-7}
&$H=50$ &35&32&2&1&2561.27\\
  \hline
 \end{tabular}
 \end{table}
From the table, we see that the {\tt FastBasis} algorithm has two advantages: (i) for randomly generated  polynomials in each class, a large proportion of them can be dealt with successfully; (ii) the average runtime of those success examples is short and a large number of  polynomials with much higher degrees, which were intractable before, can now be dealt with.


For the classes of  $n\geq40$, not too many examples are generated. In fact, dealing with ten thousands random polynomials in the class $\mathbb{Z}_{10,40}$ can take about half a month. For the class $\mathbb{Z}_{10,50}$, the corresponding runtime can be more than two months. As for the class $\mathbb{Z}_{10,60}$, the runtime can be as long as nine months. However, from the few examples, we also see that the ratio $(\#$success$\big/\#$example$)$ is high and the average runtime is acceptable.

The numerical results show that the algorithm {\tt FastBasis} can handle quite many randomly generated polynomials and it is efficient enough to solve larger problems with higher polynomial degrees, although it is not a complete algorithm. As indicated by Remark \ref{complete}, it can be modified to be a complete algorithm by combining with the algorithm {\tt FindRelations} or the algorithm {\tt  GetBasis}. Nevertheless, taking into consideration the performance of these two algorithms on the examples of small degree in Table \ref{compare}, one can expect that this completed algorithm would not be too efficient, especially for those polynomials not in the set $E$.
\section{Summary}
In this paper we propose a new sufficient condition for the exponent lattice of the roots of a polynomial in $\mathbb{Q}[x]$ to be trivial, which improves many other ones. Based on this, an algorithm is designed to compute the multiplicative relations between the roots of a  generic polynomial. The numerical results show that this algorithm can deal with a large proportion of the randomly generated  polynomials very efficiently. Moreover, it can handle many polynomials with higher degrees that are intractable by other algorithms.
\section*{Acknowledgments}
The author is very grateful to Professor Bican Xia for his helpful suggestions on the abstraction and the introduction of this paper. The author also thanks him for his advice about restating Proposition \ref{truncation} in a more general manner.
\bibliographystyle{siamplain}
\bibliography{ex_article}

\begin{thebibliography}{10}

\bibitem{2tran}
{\sc G.~Baron, M.~Drmota, and M.~Ska{\scriptsize$\L$}ba}, {\em Polynomial
  relations between polynomial roots}, J. Algebra, 177 (1995), pp.~827--846.

\bibitem{sym}
{\sc S.~D. Cohen}, {\em The distribution of the galois groups of integral
  polynomials}, Illinois J. Math., 23 (1979), pp.~135--152.

\bibitem{zariski}
{\sc H.~Derksen, E.~Jeandel, and P.~Koiran}, {\em Quantum automata and
  algebraic groups}, J. Symbolic Comput., 39 (2005), pp.~357--371.

\bibitem{dixon}
{\sc J.~D. Dixon}, {\em Polynomials with nontrivial relations between their
  roots}, Acta Arith., 82 (1997), pp.~293--302.

\bibitem{pcase}
{\sc M.~Drmota and M.~Ska{\scriptsize$\L$}ba}, {\em On multiplicative and
  linear independence of polynomial roots}, Contrib. Gen. Algebra, 7 (1991),
  pp.~127--135.

\bibitem{ge}
{\sc G.~Ge}, {\em Algorithms related to multiplicative representations}, PhD
  Thesis, University of California, Berkeley, 1993.

\bibitem{markov}
{\sc R.~Hemmecke and P.~N. Malkin}, {\em Computing generating sets of lattice
  ideals and markov bases of lattices}, J. Symbolic Comput., 44 (2009),
  pp.~1463--1476.

\bibitem{nottrans}
{\sc W.~M. Kantor}, {\em Automorphism groups of designs}, Math. Z., 109 (1969),
  pp.~246--252.

\bibitem{Kauers}
{\sc M.~Kauers}, {\em Algorithms for nonlinear higher order difference
  equations}, PhD Thesis, RISC-Linz, Linz, Austria, 2005.

\bibitem{cfinite}
{\sc M.~Kauers and B.~Zimmermann}, {\em Computing the algebraic relations of
  $\text{C}$-finite sequences and multisequences}, J. Symbolic Comput., 43
  (2008), pp.~787--803.

\bibitem{polyinv}
{\sc M.~S. Lvov}, {\em Polynomial invariants for linear loops}, Cybernet.
  Systems Anal., 46 (2010), pp.~660--668.

\bibitem{structure}
{\sc M.~S. Lvov}, {\em The structure of polynomial invariants of linear loops},
  Cybernet. Systems Anal., 51 (2015), pp.~448--460.

\bibitem{bcm}
{\sc T.~De$\;$Mazancourt and D.~Gerlic}, {\em The inverse of a block-circulant
  matrix}, IEEE Trans. Antennas and Propagation, 31 (1983), pp.~808--810.

\bibitem{generic}
{\sc I.~Rivin}, {\em Galois groups of generic polynomials}, arXiv preprint
  arXiv:1511.06446v1,  (2015).

\bibitem{sn}
{\sc C.~J. Smyth}, {\em Additive and multiplicative relations connecting
  conjugate algebraic numbers}, J. Number Theory, 23 (1986), pp.~243--254.

\bibitem{issac}
{\sc T.~Zheng and B.~Xia}, {\em An effective framework for constructing
  exponent lattice basis of nonzero algebraic numbers}, in Proceedings of the
  2019 on International Symposium on Symbolic and Algebraic Computation, ACM,
  (2019), pp.~371--378.

\bibitem{ror}
{\sc T.~Zheng and B.~Xia}, {\em An effective framework for constructing
  exponent lattice basis of nonzero algebraic numbers}, arXiv preprint
  arXiv:1808.02712v3,  (2019).

\end{thebibliography}
\end{document}